\documentclass[12pt]{article}
\usepackage{latexsym}
\usepackage{amssymb}
\usepackage{amsmath}
\usepackage{amsthm}
\usepackage{rotating}
\usepackage{multirow}
\usepackage{mathrsfs}
\usepackage{stmaryrd}
\usepackage{wasysym}
\usepackage{enumerate}
\usepackage{slashed}
\usepackage{nicefrac}
\usepackage{rawfonts}
\input{prepictex}
\input{pictex}
\input{postpictex}
\DeclareMathAlphabet{\zap}{OT1}{pzc}{m}{it}
\def\CC{\mathbb C}
\def\drc{\slashed{D}}

\DeclareMathOperator{\End}{End}

\newtheorem{thm}{Theorem}[section]
\newtheorem{prop}[thm]{Proposition}

\newtheorem{cor}[thm]{Corollary}

\newenvironment{proof1}{\medskip \noindent
{\sl Proof of Proposition \ref{uno}.}}{\hfill $\Box$
\\}
\newenvironment{proof2}{\medskip \noindent
{\sl Proof of Theorem \ref{due}.}}{\hfill $\Box$
\\}

\def\ZZ{{\mathbb Z}}
\def\RR{{\mathbb R}}
\def\CP{{\mathbb C \mathbb P}}

\DeclareMathOperator{\todd}{Td}
\DeclareMathOperator{\Ind}{Ind}

\DeclareMathOperator{\image}{image}

\DeclareMathOperator{\Tr}{Tr}
\DeclareMathOperator{\Gr}{Gr}
\DeclareMathOperator{\Hol}{Hol}
\begin{document}

\title{Einstein Constants and Smooth Topology}

\author{Claude LeBrun\thanks{Supported 
in part by  NSF grant DMS-2203572.}\\ 
Stony
 Brook
 University} 
  
\date{}

\maketitle

\hspace{1.8in}
\begin{minipage}{3.3in}
\begin{quote} {\em To my dear friend and esteemed 
 collaborator  Paul Gauduchon, in celebration   of 
his eightieth birthday. }
\end{quote}
\end{minipage}

\bigskip

\begin{abstract} 
It  was first shown   in \cite{cat} that    certain   high-dimensional smooth closed 
 manifolds    admit 
  pairs of Einstein metrics with  Ricci curvatures of opposite  sign.
After  reviewing 
 subsequent progress  that has been made on this topic, 
we   then prove various     related  results, with the ultimate  goal of   stimulating further research on associated questions.  
  \end{abstract}

\section{Introduction}

A Riemannian metric $g$ is said to be {\em Einstein} if its Ricci curvature, considered as  a function on the 
unit tangent bundle, is constant. This is of course equivalent to requiring that   the  Ricci tensor $r$ of $g$   satisfy 
\begin{equation}
\label{eineq}
r =\lambda g
\end{equation}
   for some real number $\lambda$.   In the original context of general relativity, where the metric is  taken to be Lorentzian rather than Riemannian,   equation \eqref{eineq} is called  the Einstein vacuum equation
   with cosmological constant. In fact, Einstein did  briefly consider the possibility that the constant $\lambda$ in \eqref{eineq} might 
   be non-zero, but   he eventually  \cite[p. 44]{gamow}  came to  consider  this idea to be  ``the biggest blunder
   of his life.'' It must     therefore surely be considered to be  amusingly    ironic that
 we   mathematicians have 
  acquired the habit of      
  calling $\lambda$ the {\em Einstein constant} of a Riemannian  metric $g$ satisfying \eqref{eineq}.

   In this article, the term {\em Einstein manifold}  will   mean 
   a compact connected   manifold  (always without boundary) that is equipped with a Riemannian  metric satisfying  \eqref{eineq}. 
   Of course, the modern prominence  of the study of Einstein manifolds is partly attributable    to Besse's celebrated  
   book on the subject \cite{bes}, and   many  questions  raised there  by the pseudonymous A.L.\
   Besse were  therefore quickly    elevated  to the status of 
   folk-conjectures. 
   One of these questions   \cite[p. 19]{bes}  asked whether the sign of the Einstein constant is somehow determined by the 
   diffeomorphism-type of an Einstein  manifold; that is, if a smooth closed manifold $M$ admits two Einstein
   metrics $g_1$ and $g_2$, do the Einstein constants $\lambda_1$ and $\lambda_2$ of these two metrics always  have the same sign? 
  In general,  the answer  to this question \cite{cat}  is actually a resounding, ``No!''  However, if we  refine the question by also 
  specifying   the dimension $n$ of the manifold $M$, the answer  certainly  depends on $n$, since, 
  for example, Einstein manifolds of dimension $n=2$ or $3$ automatically have constant {\sf sectional} curvature, 
  and  their Einstein constants  must therefore  have a specific, fixed sign that is predetermined by  the fundamental group. 
  This article will  emphasize
   the degree to which  we still do not know  whether or not this phenomenon really occurs in most dimensions $n$,
   while  also  highlighting some interesting related questions  that surely merit further investigation. 
  
   Even though  there  is no reason to believe  that this phenomenon  actually occurs in 
 dimension $n=4$, all of the currently-available  examples  curiously hinge on a key $4$-dimensional 
  differential-topological fact. 
  Suppose that $X$ and $Y$  are  two smooth closed oriented simply-connected   non-spin $4$-manifolds 
 that have  the same Euler characteristic and signature. In light of Donaldson's thesis \cite{donaldson}  
 and the classification of symmetric bilinear forms over the integers \cite{hm},
 a classic  result of  Terry Wall \cite{wall}  implies  that $X$ and $Y$ are then in fact   {\em $h$-cobordant}, meaning  that  there is a  
smooth  compact  oriented 
$5$-manifold-with-boundary $W$ with $\partial W = \overline{X} \sqcup {Y}$ such that the inclusions $X,Y\hookrightarrow W$
are homotopy equivalences. But this implies  that $X\times X$ and $Y\times Y$ are also $h$-cobordant, via the $9$-manifold-with-boundary  
 $W^\prime = (W\times X) \cup_{Y\times X} (Y\times W)$. Since the  simply connected-manifold $X\times X$ 
 has
 dimension  $8 > 4$, 
 Smale's $h$-cobordism 
 theorem  \cite{milnorh,smale}  therefore implies that $X\times X$ and $Y\times Y$ are diffeomorphic. Similarly,  the $h$-cobordism theorem also 
 implies  that  the $k$-fold Cartesian products $X^{\times k}= X\times \cdots \times X$ and $Y^{\times k} =Y \times \cdots \times Y$
 are diffeomorphic for every $k\geq 2$. In the examples that will interest us,  however, $X$ and $Y$ will actually \cite{FM,spccs}
 have different Seiberg-Witten invariants, implying that they are non-diffeomorphic, and 
 thereby  clearly illustrating  the failure of the $h$-cobordism theorem  in dimension $4$. Nevertheless, the $h$-cobordism
 between $X$ and $Y$ does play an important role in $4$-dimensional topology, because it provides the starting point  for 
 Freedman's proof \cite{freedman} that 
  $X$ and $Y$ are nonetheless  {\em homeomorphic}.

 Applying  this observation to   Besse's Einstein-constant question   was first made possible by the 
 discovery of a remarkable complex surface of general type by Rebecca Barlow \cite{bar}.  Her non-singular compact 
 complex surface $B$ is simply connected, with $c_1^2=1$ and $h^{2,0}=0$, and it then  follows that its underlying $4$-manifold
 is   non-spin, 
 with Euler characteristic $11$ and signature $-7$. Barlow's surface 
 is minimal, meaning that it does not contain any rational curves of self-intersection $-1$, 
 but it nonetheless contains four rational curves of self-intersection $-2$, and for this reason it does not
 have $c_1 < 0$.  However, 
 Fabrizio Catanese and I showed
 in \cite{cat}   that  generic small deformations of the complex structure of  $B$ do in fact   have $c_1 < 0$, and the 
  Aubin/Yau theorem \cite{aubin,yau} therefore implies that these generic deformations
   all carry compatible K\"ahler-Einstein metrics with $\lambda < 0$.   In particular, 
   the underlying smooth $4$-manifold 
   of the Barlow surface $B$ carries some Einstein metric $g^-$  with $\lambda < 0$,
   and all of the Cartesian products $B^{\times k}$ therefore admit an  Einstein product metric $g^-\oplus \cdots \oplus g^-$
   with the same negative Einstein constant. But the punchline is that there is also a family of {\em del Pezzo} surfaces
  $S$ that  have $c_1>0$,  are simply-connected, and also have $c_1^2=1$ and $h^{2,0}=0$.
  These specific  del Pezzo surfaces  are  diffeomorphic to  the connected sum $\CP_2 \# 8 \overline{\CP}_2$,
  and some of them  had been shown by Tian and Yau \cite{ty} to admit compatible $\lambda > 0$ K\"ahler-Einstein metrics. In particular, 
  the smooth compact $4$-manifold $S$  admits an Einstein metric $g^+$ with $\lambda >0$, and the Cartesian products 
  $S^{\times k}$ therefore admit product Einstein metrics $g^+\oplus \cdots \oplus g^+$ with the same Einstein constant $\lambda >0$.
  But since the above $h$-cobordism argument also shows that  $B^{\times k}$ and $S^{\times k}$ are actually diffeomorphic, 
   the following result \cite{cat} therefore becomes an immediate consequence:
  
  \begin{thm}[Catanese-LeBrun]
  For each $k\geq 2$, there is a closed simply-connected $4k$-manifold $M$ that admits a pair of Einstein metrics with  Einstein constants
  of opposite signs. 
  \end{thm}
  
  A series of breakthroughs in complex surface theory 
  eventually  allowed R\u{a}sdeaconu and \c{S}uvaina \cite{raresioana}  to prove a beautiful improvement of this result. First, Lee and Park \cite{park2} 
 had   discovered that
 Park's  surprising  rational-blow-down  construction \cite{parksymp} 
 of  an exotic  symplectic $4$-manifold  homeomorphic to $\CP_2 \# 7 \overline{\CP}_2$
 could be refined to produce a family of  simply-connected 
  minimal complex surface of general type with $c_1^2=2$ and $h^{2,0}=0$. Building on these ideas, Park, Park, and Shin \cite{park3}
  then produced a family of simply-connected  complex surface of general type with $c_1^2=3$ and $h^{2,0}=0$ by an analogous construction. 
  In both cases, one first produces an orbifold complex surface by collapsing  disjoint chains of rational curves  on a rational complex surface, 
  and then shows that the problem of smoothing all  these singularities  is unobstructed. What 
   R\u{a}sdeaconu and \c{S}uvaina then showed was that the generic   surface 
      in each smoothing   family  has $c_1 < 0$. Park, Park, and Shin \cite{park4} then went on to produce 
    a similar family of  simply-connected minimal complex surface of general type with $c_1^2=4$ and $h^{2,0}=0$, and   then 
   checked that 
   the R\u{a}sdeaconu-\c{S}uvaina argument also shows that the generic surface in their family  has $c_1 < 0$. By once again
   invoking the Aubin-Yau theorem, and remembering our previous discussion of the Barlow case, 
   we thus obtain four smooth simply-connected closed $4$-manifolds $X_\ell$, $\ell = 1, \ldots, 4$, each of which carries 
   an Einstein metric $g_\ell^-$ with $\lambda =-1$ that  is K\"ahler with respect to some complex structure  with  $c_1^2=\ell$ and $h^{2,0}=0$. 
   On the other hand, Wall's theorem tells us that each $X_\ell$ is $h$-cobordant to the smooth oriented
   manifold  $Y_\ell := \CP_2 \# (9-\ell)\overline{\CP}_2$ via some compact oriented $5$-manifold-wth-boundary $W_\ell$ with  $\partial W_\ell =
   \overline{X}_\ell \sqcup Y_\ell$. It then follows that the products $X_{\ell_1}\times \cdots \times X_{\ell_k}$ 
   and $Y_{\ell_1}\times \cdots \times Y_{\ell_k}$ are $h$-cobordant via the manifold-with-boundary 
    obtained by gluing  the $k$ different $h$-cobordisms
  $$X_{\ell_1}\times \cdots \times X_{\ell_{j-1}}\times W_{\ell_j} \times Y_{\ell_{j+1}}\times  \cdots \times Y_{\ell_k}$$
  end-to-end, and  the $h$-cobordism theorem therefore 
   gurantees that these products 
   are 
  actually  diffeomorphic, for  any choice of $\ell_1, \ldots ,\ell_k$, where $k\geq 2$.
   However, Tian and Yau \cite{ty}  had also shown  that each such 
    $Y_\ell$ is diffeomorphic to a del Pezzo surface that admits  a compatible K\"ahler-Einstein metric $g_\ell^+$ with $\lambda =+1$. (Indeed, 
     we  now  know \cite{sunspot} that all del Pezzo surfaces of these diffeotypes  actually admit K\"ahler-Einstein metrics.) 
  It thus follows that any product $X_{\ell_1}\times \cdots \times X_{\ell_k}$ admits a $\lambda=-1$ Einstein metric 
  $g_{\ell_1}^-\oplus \cdots \oplus g_{\ell_k}^-$, while any  product $Y_{\ell_1}\times \cdots \times Y_{\ell_k}$ admits a $\lambda=+1$ Einstein metric 
  $g_{\ell_1}^+\oplus \cdots \oplus g_{\ell_k}^+$. But since these products have already been observed to be  diffeomorphic, we therefore  obtain
  \cite{raresioana}
   the following result :

  \begin{thm}[R\u{a}sdeaconu-\c{S}uvaina] For every $k \geq 2$, there are at least $\binom{k + 3}{3}$ distinct smooth closed simply-connected $4k$-manifolds
  that admit both $\lambda > 0$ and $\lambda < 0$ Einstein metrics. 
    \end{thm}
    
 \begin{proof} If we choose any $k$ smooth $4$-manifolds $Y_{\ell_j}=\CP_2 \# (9-\ell_j)\overline{\CP}_2$,  where  the integers $\ell_j$ satisfy 
 $1\leq \ell_j \leq 4$ for 
 every $j=1,\ldots , k$,  we have just seen that their Cartesian product is 
 a smooth compact simply-connected $4k$-manifold that admits both $\lambda > 0$ and $\lambda < 0$
 Einstein metrics. But the chosen integers $\ell_j$  are also determined, up to permutations, 
    by  the  Betti numbers of the product manifold. Indeed,  the Poincar\'e polynomial 
  $\sum b_i t^i$ is multiplicative under Cartesian products, 
  while  the Poincar\'e polynomial  $1+(10-\ell_j) t^2 +t^4$ of each  $Y_{\ell_j}=\CP_2 \# (9-\ell_j)\overline{\CP}_2$
   is
  irreducible in $\ZZ [t]$. The fact that   $\ZZ [t]$ is a unique factorization domain  thus implies that   the 
  ${\ell_j}$ are determined, up to permutations,  by   the Betti numbers of the Cartesian  product.  \end{proof}

These examples clearly show 
that  $\lambda > 0$ and $\lambda < 0$ 
Einstein metrics can  definitely coexist on various   smooth compact  manifolds. But 
the list of manifolds where this phenomenon has actually been seen  to occur remains surprisingly limited, 
and consists of spaces  than share many    rare  and unusual features. 
We  can  therefore   probably  expect that
many  of the  patterns seen in  these known 
examples 
will   be nothing but artifacts of  the method of construction. For example, all of these known   examples 
  live in dimensions $\equiv 0 \bmod 4$. Do  analogous things also happen  in odd
dimensions? Or in dimensions $\equiv 2 \bmod 4$? The known constructions also never involve
any Einstein metrics with $\lambda=0$. 
What can actually be proved   about the coexistence of  Ricci-flat
metrics with  Einstein metrics for which  $\lambda \neq  0$? 
The bulk of this article will actually concern  
the narrower problem of finding closed manifolds 
that  admit both Ricci-flat metrics and $\lambda > 0$ Einstein metrics, with a particular emphasis on the 
situation in  dimensions $6$ or $7$. One motivation for
this focus is that there are many simply-connected  Ricci-flat manifolds of special holonomy 
in these dimensions, and yet  every simply-connected closed manifolds of dimension $6$ or $7$ 
also admits \cite{gvln} metrics of positive scalar curvature. This narrower problem turns out to be surprisingly delicate,  however, 
and many of the results we will prove will be  negative in nature. Nonetheless, 
it is hoped that describing some of these results here will have a salutary effect, 
both by encouraging other researchers to join the search, and by  charting   various 
reefs and shoals  that  future explorers  might do well to avoid.

\vfill 
\noindent
{\bf Acknowledgments.}
The author  warmly thanks  Johannes Nordstr\"om for some  illuminating   conversations   that helped put  him on the road
  to   the results of   \S \ref{g2se}. He would also like to thank   Charles Boyer,   Rare\c{s} R\u{a}sdeaconu, and Ioana \c{S}uvaina  
for   a number of    stimulating  discussions of relevant topics.

 \pagebreak 
 
\section{$G_2$ Manifolds  and Sasaki-Einstein Metrics} 
\label{g2se} 

One plausible-sounding strategy for finding odd-dimensional avatars  of the phenomenon under
discussion might  be to look for  a simply-connected compact $7$-manifold that admitted both   a Sasaki-Einstein metric  $g_1$
and a 
metric $g_2$ of holonomy $G_2$. Indeed, if this were  possible, 
 $g_2$ would then  be 
Ricci-flat, while $g_1$ would   be an Einstein metric with  positive Einstein constant.
Since established  methods are  available \cite{bg,chn-twisted,joycebook} for manufacturing vast numbers   of compact $7$-dimensional
Einstein manifolds of both these  types, one might therefore be tempted to  begin a search 
for 
 diffeomorphic pairs of Einstein $7$-manifolds of these different types. However,  any such search would be
 doomed   to fail. Indeed, in this section, we will prove the following result:

\begin{thm} \label{due}
No smooth compact $7$-manifold can admit both a Sasaki-Einstein metric $g_1$
and  a metric $g_2$ with   holonomy contained in $G_2$. 
\end{thm}

Our proof of the incompatibility of these two types   of Einstein metric   is  based on the following topological fact:

\begin{prop} \label{uno} 
If $M$ is a   smooth compact $7$-manifold that carries a Sasaki-Einstein metric $g_1$, 
then $p_1 (M)\in H^4(M,\ZZ)$ is a torsion class. 
\end{prop}

Before we prove this in full generality, let us first  see    why  it  holds 
for  the prototypical  examples that were originally  discovered  \cite[Theorem 5]{kobs-se} by Shoshichi Kobayashi.
These are exactly  the   Sasaki-Einstein manifolds $(M,g_1)$  which   are   {\em regular} \cite{bg}, in the sense     that 
the Reeb vector field $\xi$ is periodic, with all orbits of the same  minimal period. 
This periodicity assumption  forces  $M$ to be  a principal circle bundle over a compact $6$-manifold 
$$ \begin{array}{ccc}
    S^1  & \to &M^7\\
    &   &\,  \downarrow^{\varpi} \\
    & & X^6
\end{array}
$$
and the quotient space $X=M/S^1$ then inherits a uniquely-determined metric $h$ that makes $\varpi: M \to X$ into a 
Riemannian submersion. With  standard normalization conventions,  the Sasaki-Einstein metric $g_1$ then has Einstein constant
$6$, while  the Riemannian-submersion metric $h$  induced on $X$ by  via $\varpi$ is  K\"ahler-Einstein, 
with Einstein constant $8$. 

By passing to a  finite cover if
necessary,  we may now arrange for $M$ to be simply connected. The circle bundle 
$M\to X$ can then be identified with the unit-circle bundle of $K^{-1/m}$, where $K=\Lambda^{3,0}_X$ is the canonical line-bundle for 
an $h$-compatible complex structure $J$ on $X$, and where the  Fano index $m\in \{1, 2, 3 ,4\}$ of $(X,J)$ is by definition 
   the largest integer that
divides $c_1(X)\in H^2 (X, \ZZ )$. With standard  normalization conventions, each fiber of $\varpi$  is then    a 
closed geodesic in $(M,g_1)$ of total length $m \pi/2$. 

We next compare the cohomologies  of $M$ and $X$ by invoking   the Gysin exact sequence \cite[\S 6]{milnorstaf}
$$\cdots \to H^k (M) \to H^{k-1} (X) \stackrel{\mathsf{L}}{\to} H^{k+1} (X) \stackrel{\varpi^*}{\to} H^{k+1} (M)\to \cdots$$
where $\mathsf{L}$ is given by     cup product with the Euler class $c_1(X)/m$ of the oriented circle bundle $M \to X$.
 Notice, however,  that  $c_1(X)/m$ is  also  the K\"ahler class of $4h/m\pi$.   In deRham cohomology, this means
that  $\mathsf{L}$  actually becomes  the {\em Lefschetz operator}  of a  K\"ahler metric on   the compact complex $3$-fold 
 $(X,J)$. The 
{\sf Hard Lefschetz theorem} \cite[\S 0.7]{GH} therefore guarantees that 
$$\mathsf{L} : H^2 (X,\RR) \to H^4 (X,\RR)$$
is a vector-space isomorphism. The real-coefficient version 
$$\cdots \to H^{2} (X,\RR) \stackrel{\mathsf{L}}{\to} H^{4} (X,\RR) \stackrel{\varpi^*}{\to} H^{4} (M,\RR)\to \cdots$$
of the Gysin sequence consequently   guarantees  that the pull-back map 
$$\varpi^* : H^{4} (X,\RR) \to H^{4} (M,\RR)$$ 
must vanish  identically.

In particular, the pull-back  $\varpi^*p_1(X)\in H^4(M, \ZZ)$ of the first Pontrjagin  class of  $X$ must be a torsion class, since its image
$\varpi^*p_1^\RR(X)\in H^4(M, \RR)$  in deRham cohomology necessarily  vanishes. 
But since $\ker \varpi_*\subset TM$ is just the trivial rank-$1$ bundle
spanned by the Reeb vector field, $TM\cong \RR \oplus \varpi^* TX$, and it therefore  follows that 
$$p_1 (M) = p_1 (TM) = p_1 (\RR \oplus \varpi^* TX) =  p_1( \varpi^* TX ) = \varpi^* p_1(TX) = \varpi^* p_1(X)$$
is therefore   a torsion class, as claimed.  \hfill $\diamondsuit$

\bigskip

To  prove  the general case of Proposition \ref{uno}, we will now  replace the  cohomology  of $X$  with the basic cohomology 
 of  the Reeb foliation of $M$.

\begin{proof1} Recall that a Sasaki-Einstein manifold $(M,g_1)$  is  an Einstein manifold  
that is equipped    \cite[\S 2.1]{bg-surv} 
 with a unit-length  Killing field  $\xi$
such  that  $(M,g_1)$  has  sectional curvature   $K(P)=1$ for every  $2$-plane $P\subset TM$ with $\xi\in P$, and  where the   
simple $2$-form representing any such oriented   $2$-plane $P$ is  moreover an eigenvector of the curvature operator. The $1$-form
$\eta = g_1(\xi, \cdot )$  is then a contact form, and    $\xi$  is then the  Reeb vector field  of the scaled contact manifold
$(M,\eta )$. Setting $\omega = d\eta/2$, and letting $\Pi  : TM \to \ker \eta$ denote projection modulo the span of $\xi$, 
we then have 
$$g_1 = \eta \otimes \eta + \omega (\,\cdot\, , (J \circ \Pi) \,\cdot\,)$$
for a unique CR structure $J \in \End (\ker \eta )$, $J^2 = -I$, that is moreover   invariant under the 
flow of $\xi$. Letting  $\mathfrak{F}$ denote  the foliation of $M$  by the  flow-lines of   $\xi$,
the {transverse metric} $h$ locally induced on the leaf space $M/\mathfrak{F}$ 
is  then K\"ahler-Einstein, with complex structure induced by $J$. We will henceforth assume that 
$M$ has real dimension $7$, which then implies  that $g_1$ has 
Einstein constant $6$, and that  $h$ has Einstein constant $8$.

 Relative to the Reeb foliation $\mathfrak{F}$, a differential $p$-form $\alpha \in \mathcal{E}^p(M)$ is said
 to be {\em basic} if $\xi\lrcorner\, \alpha = 0$ and $\xi\lrcorner\, d\alpha = 0$. The basic forms $\mathcal{E}_B^*(M, \mathfrak{F})$
 then constitute  a sub-complex 
 $$\cdots \stackrel{d}{\to}  \mathcal{E}_B^{p-1}(M, \mathfrak{F})\stackrel{d}{\to} \mathcal{E}_B^{p}(M, \mathfrak{F})\stackrel{d}{\to} \mathcal{E}_B^{p+1}(M, \mathfrak{F})\stackrel{d}{\to} \cdots$$
of the deRham complex $\mathcal{E}^*(M)$,  and the basic cohomology of $(M, \mathfrak{F})$ can therefore be  defined to be 
$$H^p_B (M, \mathfrak{F}) := \frac{\ker \left[d: \mathcal{E}_B^{p}(M, \mathfrak{F})\to \mathcal{E}_B^{p+1}(M, \mathfrak{F})\right]}{\image  \left[d: \mathcal{E}_B^{p-1}(M, \mathfrak{F})\to\mathcal{E}_B^{p}(M, \mathfrak{F})\right]}~.$$
If a basic $p$-form
$\alpha$ is  closed, we will use $\llbracket \alpha \rrbracket$ to denote its equivalence class in $H^p_B (M,\mathfrak{F})$,
while  instead using $[\alpha]\in H^p_{dR}(M)$ to  denote its (larger)  equivalence class in ordinary deRham cohomology.

Because $M$ is compact and $\mathfrak{F}$ is a Riemannian foliation, the basic cohomology  $H^*_B (M, \mathfrak{F})$ is  \cite{eka-fin}
automatically   finite-dimensional. Moreover, because the complex dimension of 
$M/\mathfrak{F}$ is  $3$, and  the transverse metric $h$ is 
K\"ahler, with  the  basic form $\omega$ as its K\"ahler form,
 El Kacimi-Alaoui's Hard Lefschetz Theorem \cite[\S 3.4.7(i)]{eka-lef} guarantees that the linear map 
  \begin{eqnarray*}
\mathsf{L} : H^2_B (M, \mathfrak{F}) & \to & H^4_B (M, \mathfrak{F}) \\
 \llbracket \alpha \rrbracket &\mapsto & \llbracket \omega \wedge \alpha \rrbracket
\end{eqnarray*}
is an isomorphism. 

On the other hand, 
notice that the first Pontrjagin class of $TM/T\mathfrak{F}$ 
is  represented in deRham cohomology by the closed basic $4$-form
$$\wp_1 (TM/T\mathfrak{F}) =  -\frac{1}{8\pi^2} \Tr (\mathscr{R}\wedge \mathscr{R})$$
where $\mathscr{R}$ denotes the Riemann curvature tensor of the transverse metric $h$. 
We  can thus define  a basic version  of the first Pontrjagin class $TM/T\mathfrak{F}$   by
$$p_{1,B}^\RR(TM/T\mathfrak{F}) :=\llbracket -\frac{1}{8\pi^2} \Tr (\mathscr{R}\wedge \mathscr{R})\rrbracket \in H^4_B (M, \mathfrak{F}).$$
By the above version of Hard Lefschetz, we must therefore have 
$$p_{1,B}^\RR(TM/T\mathfrak{F}) = \llbracket \omega \wedge \alpha \rrbracket$$
for some closed basic $2$-form $\alpha$. It therefore follows that 
$$-\frac{1}{8\pi^2} \Tr (\mathscr{R}\wedge \mathscr{R}) = \omega \wedge \alpha + d\beta$$
for some basic $3$-form $\beta$ on $M$. But since we can identify $TM/T\mathfrak{F}$ with $\ker \eta$,
the usual    first Pontrjagin class of $\ker\eta$    in  deRham cohomology must be 
$$p_1^\RR(\ker \eta )=p_1^\RR(TM/T\mathfrak{F}) = [-\frac{1}{8\pi^2} \Tr (\mathscr{R}\wedge \mathscr{R})]= [\omega \wedge \alpha ]
= [d(\textstyle{\frac{1}{2}}\eta \wedge \alpha)] =0,$$
and it therefore follows that 
$$p_1^\RR (M)= p_1^\RR (TM)=p_1^\RR ( \RR \oplus \ker \eta ) = p_1^\RR ( \ker \eta ) =0\in H^4_{dR}(M) = H^4(M, \RR).$$
Hence $p_1(M) \in H^4(M, \ZZ )$ must be a torsion class, as claimed. 
\end{proof1}

With this result in hand, it now becomes straightforward to prove the main result of this section:

\begin{proof2} If  the closed $7$-manifold  $M$ admits a Sasaki-Einstein metric $g_1$,  
 the positivity of the Ricci curvature of $g_1$ forces  $\pi_1(M)$ to  be  finite by Myers' Theorem. On the other hand, 
 Proposition \ref{uno} 
tells us that $p_1(M)\in H^4(M, \ZZ)$ must be  a torsion class. 

But now suppose that  $M$ also admits a metric $g_2$ with holonomy contained in $G_2$. Any $g_2$-compatible choice of torsion-free $G_2$-structure then determines a calibrating $3$-form $\varphi$ that is parallel with respect to $g_2$.
In particular, this   $\varphi$ is  closed, and so  
endows $M$ with a preferred deRham class $[\varphi]\in H^3_{dR}(M)$.
  However, a remarkable 
formula of Joyce \cite[Proposition 10.2.7]{joycebook} then tells us that 
\begin{equation}
\label{rejoyce} 
\langle p_1(M)\cup [\varphi ], [M]\rangle = - \frac{1}{8\pi^2} \int_M |\mathcal{R}|^2 d\mu_{g_2},
\end{equation}
where $\mathcal{R}$ denotes the Riemann curvature tensor of $g_2$. 
But since  $p_1(M)$ has been shown to be  a torsion class in $H^4(M, \ZZ)$, both  sides of \eqref{rejoyce} must therefore vanish. 
Hence  $g_2$ must   be  flat, and the universal cover of $(M,g_2)$ must  be  $\RR^7$.
 Since $M$ is compact, this implies  that 
 $\pi_1(M)$ must be  infinite. But    we have already shown  that $\pi_1(M)$
is finite,  so this is  a contradiction! Thus,   the purported coexistence of the    metrics $g_1$ and $g_2$
is actually  impossible. 
 \end{proof2}

 While  this does not by any  means  prove that a $G_2$ manifold $(M^7, g_2)$ can never carry an Einstein metric $g_1$ 
 with $\lambda > 0$,
 it does at least rule out one way that $g_1$ could be related to metrics of special holonomy. We now close this section
 by eliminating  various other simple ways  in which  such a metric $g_1$ might  arise from a  special-holonomy construction. 
 
 \begin{prop} If a smooth  compact connected  
 $7$-manifold $M$ admits  a  $\lambda> 0$  Einstein metric $g_1$, and also admits a metric $g_2$ 
 with holonomy in $G_2$, then the holonomy group of $(M, g_1)$ must be  exactly    $SO(7)$. 
 \end{prop}
\begin{proof} By applying Myers' theorem to the universal cover $(\widetilde{M}, \tilde{g}_1)$  of $(M, g_1)$, we immediately see  that $\widetilde{M}$ is also compact. It will therefore suffice to prove the claim when $M$ is  simply connected, 
since the holonomy group  $\Hol (\widetilde{M}, \tilde{g}_1)=\Hol^0(M,g_1)$  would   in any case     be 
a subgroup of  $\Hol (M,g_1)\subset O(7)$, while 
the existence of $g_2$ would guarantee   that $M$ is orientable, thereby implying 
  that $\Hol (M,g_1)\subset SO(7)$.

With this simplification, Berger's classification theorem \cite{berger-hol,bes,simons-hol}  implies that  our simply-connected, positive-$\lambda$ Einstein $7$-manifold 
$(M,g_1)$  must have holonomy $G_2$ or $SO(7)$  unless it is  either a non-trivial Riemannian product
or an irreducible  symmetric space of compact type. But, aside from the $5$-dimensional special-Lagrangian Grassmannian $SU(3)/SO(3)$,
 any  simply-connected, compact-type 
 irreducible symmetric space of   prime dimension $p$ is a  
  round $p$-sphere, with  holonomy $SO(p)$;
  cf. \cite[\S 10.K, Tables 2 \& 4]{bes} or \cite{helgason}. Likewise, we can  immediately rule out  $G_2$  as a candidate for 
 $\Hol (M, g_1)$, since 
 $g_1$ is assumed to have  non-zero  Ricci curvature. To prove the claim, it will therefore suffice to  prove that 
 $(M,g_1)$ cannot be a non-trivial Riemannian product.

However, a non-trivial   Riemannian-product decomposition of the simply-connected positive-$\lambda$ Einstein $7$-manifold 
$(M,g_1)$ would  express the space as a product of two positive-$\lambda$ Einstein manifolds, one  of which would
necessarily have dimension $< 4$. But the only  simply-connected positive-$\lambda$ Einstein manifolds of dimension $< 4$
are the  round spheres of dimension $2$ and $3$. Thus, such a decomposition would necessarily take the form 
$$
(M^7,g_1) = \begin{cases}
    (S^2, g_S) \times (X^5, g_X),  & \text{or} \\
    (S^3, g_S) \times (Y^4, g_Y),   &
\end{cases}
$$
where $g_S$ is some  constant multiple of the unit-sphere  metric, and where  $g_X$ or $g_Y$  would be a 
positive-$\lambda$ Einstein metric. On the other hand, we have assumed that $M$ also admits a metric $g_2$
with holonomy in $G_2$, so  $X$ or $Y$ would necessarily be spin. Moreover, since $M$ is assumed 
to be simply connected, $g_2$ cannot be flat, and \eqref{rejoyce}  therefore implies that  $M$ must  admit a cohomology class  
$[\varphi ]\in H^3_{dR} (M)$ such that
 $[\varphi ] \cup p_1(M) \neq  0$. We will now use these observations to obtain a contradiction in either case. 
 
 First, if $M = S^2 \times X^5$, the fact that  $TS^2$ is stably trivial immediately implies that 
 $p_1(M) = \varpi^* p_1(X)$, where $\varpi: M \to X$ is the second-factor 
 projection.  But  since  the  simply-connected $5$-manifold $X$ certainly  satisfies  $H^1_{dR}(X)=0$, 
  the K\"unneth formula $H^*_{dR}(M) = H^*_{dR}(S^2) \otimes H^*_{dR} (X)$  
  tells us that 
  $\varpi^* : H^3_{dR} (X) \to H^3_{dR} (M)$ is
  an isomorphism. Consequently,   every  $[\varphi] \in H^3_{dR} (M)$ can be expressed as $[\varphi ] = \varpi^* [\phi ]$ 
 for some $[\phi ]\in H^3_{dR} (X)$. It then follows that   $p_1(M) \cup [\varphi ] = \varpi^* ( p_1 (X) \cup [\phi ] ) = \varpi^* 0 = 0$. Since this holds  for 
 any $[\varphi ]\in H^3_{dR} (M)$, no such product $S^2 \times X$  could   possibly admit a metric $g_2$ with holonomy in $G_2$. 
 
 On the other hand, if   $M= S^3 \times Y^4$,   recall  that the existence of a metric $g_2$   with holonomy in $G_2$    forces $M$, and hence 
  $Y$,  to be spin. But since $TS^3$ is trivial,  we also   have $\varpi^* p_1(Y)= p_1(M) \neq 0$, where 
 $\varpi: M\to Y$ denotes the second-factor projection. Hence $p_1(Y) \neq 0$, and it 
  therefore follows that the closed spin $4$-manifold $Y$ has  $\widehat{A} (Y ) = -\frac{1}{24}\langle p_1 (Y), [Y]\rangle \neq 0$.
 But since $\widehat{A} (Y )$ is actually  the index of the chiral Dirac operator $\slashed{D}: \Gamma (\mathbb{S}_+)\to \Gamma (\mathbb{S}_-)$
 on $Y$, 
the Lichnerowicz Weitzenb\"ock 
formula then  guarantees that the spin  $4$-manifold $Y$ cannot  admit a metric $g_Y$ of positive scalar curvature \cite{hitharm,lawmic,lic}. This contradiction now shows that  the positive-$\lambda$ Einstein metric 
$g_1$ must  in fact be holonomy irreducible, and  so must  actually  have holonomy $SO(7)$. \end{proof}

\section{Fano Manifolds and Calabi-Yau $3$-Folds}

In dimension $7$, 
we have just seen that  Ricci-flat metrics of special holonomy cannot coexist with $\lambda >0$ Einstein 
metrics of Sasakian type.  But where else,  besides dimension $7$,   might we look for a smooth compact manifold $M^n$  that admitted 
 both a Ricci-flat metric $g_0$ and 
a positive-$\lambda$ Einstein  metric $g_1$? If, for ease of construction, 
we  also required   the Ricci-flat metric $g_0$ to  have special  holonomy and to be  holonomy irreducible, 
 then any  
candidate for $M$ 
would  consequently 
have  dimension $n \equiv 6 \bmod 8$. Indeed, after  arranging for $M$ to be simply connected by,  if necessary, 
  passing to a finite cover, the Ricci-flat metric $g_0$  would   have 
 holonomy $SU (m)$ (with $n=2m$),   $Sp (k)$ (with $n=4k$), or $Spin (7)$ (with $n=8$). 
 Because this  reduces the structure-group of $M$ to a simply-connected  Lie group,
 our manifold would therefore   be spin, and, since $\pi_1(M)=0$,  its  spin structure would moreover be unique.
 Because  $n=\dim M$ is even, we therefore obtain a chiral 
 Dirac operator  $$\slashed{D}: \Gamma (\mathbb{S}_+)\to \Gamma (\mathbb{S}_-)$$  
 associated with each metric on the manifold. 
 However, because a section of $\mathbb{S}_+\oplus \mathbb{S}_-$  on a compact, even-dimensional,
 Ricci-flat spin manifold $(M,g_0)$ belongs   to $\ker \slashed{D}\oplus \ker \slashed{D}^*$  if and only if it is parallel \cite{hitharm,lic},
 the generalized index of $\slashed{D}$ can therefore be read off by counting the trivial factors of  the spin representation of the holonomy group
 $\Hol (M,g_0)$. 
When $n\equiv 0\bmod 4$, the classical index $\widehat{A}(M)=\Ind (\slashed{D})=\dim \ker (\slashed{D})-\dim \ker (\slashed{D}^*)$ is thus  
$$
\widehat{A}(M) = \begin{cases}
     2, & \text{if  } \Hol(g_0) = SU(2m) \text{ and }  n=2m, \\
       k+1, & \text{if  } \Hol(g_0) = Sp(k)  \text{ and } n=4k, \\
    1, & \text{if  } \Hol(g_0) = Spin  (7)  \text{ and } n=8.
\end{cases}
$$
On the other hand,  for $n = 8\ell + 2$, our Ricci-flat metric $g_0$ would have  holonomy  $SU (4\ell +1)$, 
and the same reasoning then  yields  $\dim \ker \drc =1$; and  since Hitchin's $\alpha$-invariant is exactly the mod-$2$
reduction of $\dim \ker \drc$ when $n\equiv 2 \bmod 8$, we similarly conclude that $\alpha (M)\neq 0$ in this case, too. 
 In each  of these cases, 
 it therefore follows \cite{hitharm} that the manifold cannot admit metrics of positive scalar curvature, which in particular 
 precludes the existence of a  $\lambda >0$ Einstein metric. This leaves us with only 
  the case where the Ricci-flat metric has  holonomy $SU (4\ell +3)$,
 and lives on a manifold of real dimension $8\ell +6$. On the other hand, this remaining case certainly has a much more encouraging feel to it, 
 because   the cobordism-based results  of   Gromov-Lawson \cite{gvln} and Stolz \cite{stolz} 
 show  that  any simply-connected manifold of dimension $6 \bmod 8$   does actually  admit
 Riemannian metrics of positive scalar curvature. 
 
 While these arguments provide only very  rough guidance that is far from    watertight, 
 they at least   suggest that we  might do well to begin our  search in  real  dimension $6$.
The specific problem we will therefore  focus on in this section  is that of  finding a smooth closed $6$-manifold $M$ that  admits both 
a Ricci-flat K\"ahler metric  $g_0$ and a positive-$\lambda$ K\"ahler-Einstein  metric $g_1$, where it is to  be
understood from the outset that the complex structures associated with these two K\"ahler structures
would  be  entirely unrelated. 
To get a handle on the problem, we will also invoke   a beautiful  result of  Terry Wall \cite[Theorem 5]{wall6} that reveals
the non-existence of 
 exotic differential structures in dimension $6$  by
 determining the diffeomorphism-type of a large class of  $6$-manifolds in terms of  a small number of classical invariants.

Now recall that  a compact complex  manifold  $(X,J_X)$ is said to be   {\em Fano} if  $c_1 (X,J_X)  > 0$. By
Yau's proof \cite{yauma} of the Calabi conjecture, this is equivalent to saying that $(X,J_X)$  admits a compatible K\"ahler metric
of positive Ricci curvature. Since  the Kodaira vanishing theorem implies  that the Todd genus $\todd (X)=\chi (X, \mathcal{O})$ of
any Fano manifold  
must equal $1$, and since the Todd genus is   multiplicative under finite covers,
it follows \cite{kobs-se} that any Fano manifold is necessarily simply connected. On the other hand,  while  any complex manifold that admits 
a compatible  $\lambda >0$ K\"ahler-Einstein  metric is necessarily Fano,  the converse is certainly  false. Instead, 
a landmark result of Chen, Donaldson, and Sun \cite{cds0,cds1,cds2,cds3} says that a Fano manifold  admits a compatible 
($\lambda > 0$) K\"ahler-Einstein  metric if and only if it satisfies the subtle  algebro-geometric condition of {\em $K$-polystability}.  Fortunately, though, 
Fano manifolds of complex dimension $3$ 
 have been completely classified \cite{iskov1,iskov2,momu2,momu1}. This has become  a springboard for   systematic studies 
of their $K$-stability \cite{abzhfano,fano3cal,arezzofano,dervanfano,liuxufano,spotsunfano}, and these studies     will  provide us with all the information about existence that will be required    for our  purposes.

Of course,  we will only be interested  here in those  smooth  $6$-manifolds underlying Fano $3$-folds  that  might also
simultaneously   support a Calabi-Yau  metric $g_0$
that is compatible with some other complex structure $J_0$.
Since we've just seen that the existence of a positive-$\lambda$ K\"ahler-Einstein metric $g_1$  forces $M$ to be simply connected, 
the flat canonical line-bundle of $(M,g_0,J_0)$ must  therefore   be trivial, and  the mod-$2$ reduction 
$w_2 (M)\in H^2(M, \ZZ_2)$ of $c_1(M, J_0)=0$  consequently 
 vanishes. Thus   $M$ must necessarily   be spin. If $J_1$ denotes a complex structure compatible 
with $g_1$,  the mod-$2$ reduction of $c_1(M,J_1)$  therefore  also vanishes, and  the exact sequence
$$\cdots \to H^2 (M, \ZZ) \stackrel{2\cdot}{\to} H^2 (M, \ZZ) \to H^2 (M, \ZZ_2) \to \cdots
$$  
 then
 tells us   that   $2| c_1(M,J_1)$. Thus, the {\em Fano index} of $(M,J_1)$ must be even, where the Fano
index  of a Fano manifold is by definition the largest integer that divides  the first Chern class. 
On the other hand, a general result of Kobayashi-Ochiai \cite{kofanoindex} implies that 
the Fano index of  a Fano $3$-fold $(M,J_1)$  is necessarily  $\leq 4$, with equality iff $(M,J_1)= \CP_3$.
However, since $(M,J_0)$ has already been observed to have trivial canonical line bundle, we then have 
 $h^{3,0}(M,J_0)=1$, and hence   $b_3 (M)\geq 2$ by the Hodge decomposition. Thus 
  $M$ certainly cannot be diffeomorphic to  $\CP_3$, and   the Fano index of any Fano candidate $(M,J_1)$ is therefore exactly $2$.

 But these same ideas also rule out several other possibilities. 
 Indeed, by the classification results of \cite{momu2,momu1}, or by a more general result of Wi\'sniewski \cite{wis}, 
 the only index-$2$ Fano $3$-folds
 with $b_2\geq 2$ are $\mathbb{P}(T^*\CP_1)$, the one-point blow-up of $\CP_3$, and $\CP_1 \times \CP_1 \times \CP_1$;
 these  are often identified by the  standard catalog numbers 2-32, 2-35, and 3-27, respectively, where the  digit before the dash
 records the second Betti number $b_2$. 
 Since these are all $\CP_1$-bundles over $\CP_2$ or $\CP_1\times \CP_1$, the Serre spectral sequence 
  shows that they all satisfy $b_3=0$, and this of course  rules out 
 the existence of  Calabi-Yau metrics $g_0$ on all of 
  their underlying $6$-manifolds.
   
 In addition, there is one last  deformation-type of degree-$2$ Fano $3$-folds  
that is also  ruled out by this same argument. This candidate, which bears the conventional 
 catalog number 1-15, is the transverse intersection of a linear $\CP_6\subset \CP_9$ with the image of the Pl\"ucker embedding
 $\Gr (2, 5)\hookrightarrow \mathbb{P} (\Lambda^2 \CC^5) = \CP_9$. Since $\Gr (2, 5)= U(5) /[U(2)\times U(3)]$ has 
 $H_2=\pi_2 = \ZZ$ by the homotopy exact sequence of a fibration, the Lefschetz hyperplane section theorem shows 
 that this Fano $3$-fold  must indeed have $b_2=1$, as is implicitly claimed by  the first digit of its catalog number. 
 However, a Chern-class computation also reveals that this Fano manifold
 has Euler characteristic $c_3=4$, and we  therefore deduce that it also  has $b_3=0$. Thus, 
 the existence of a Calabi-Yau metric  on its underlying $6$-manifold is automatically  excluded. 
 
 Thus, we have failed to exclude exactly four deformation-types of Fano $3$-folds, namely those with 
Fano index $2$,  $b_2=1$, and $b_3\neq 0$. Moreover, every Fano  $3$-fold in each of these four families is actually 
$K$-stable by the results of 
\cite{abzhfano,arezzofano,dervanfano,liuxufano,spotsunfano}, and therefore admits a K\"ahler-Einstein metric; see  \cite{fano3cal} for an excellent overview, and further references. The following Proposition articulates  this conclusion more explicitly, and in further detail:
 
  \begin{prop}\label{quidnunc} 
  If $(M,J_1,g_1)$  is a $\lambda > 0$ K\"ahler-Einstein  manifold of complex dimension $3$ whose 
   underlying  $6$-manifold $M$ also admits a Calabi-Yau metric $g_0$ (compatible with some other complex structure
  $J_0$), then the Fano $3$-fold $(M,J_1)$  belongs to one of the four   following  families:
  \vspace{-.2in}
 \begin{center}
 \begin{tabular}{|c|c|c|c|c|l|}
\hline
ID$\#$ &  $c_1^3$  & $h^{1,2}$ & $h^{1,1}$  & Index&Description \\
   \hline 
1-11   &  8&  21&1&2&sextic  hypersurface in $\CP (1,1,1,2,3)$\\
   \hline    
   1-12&  16& 10 &1&2&double cover of $\CP_3$ branched over  a smooth quartic\\
     \hline 
 1-13 &  24&  5&1&2&cubic hypersurface in  $\CP_4$ \\
   \hline  
   1-14 &  32& 2&1&2& transverse intersection of two quadrics in $\CP_5$\\
\hline
\end{tabular}
\end{center}
Moreover, every Fano manifold in any of these families is $K$-stable, and therefore admits a compatible 
$\lambda > 0$ K\"ahler-Einstein metric $g_1$. 
\end{prop}

Inspection now also  reveals  the following useful topological fact:

\begin{cor} Every $6$-manifold $M$ in Proposition \ref{quidnunc} satisfies $$\mathfrak{T}_2= \mathfrak{T}^3 = \mathfrak{T}_3=\mathfrak{T}^4=0,$$
where  $\mathfrak{T}_j$ denotes the torsion subgroup of $H_j(M, \ZZ)$, and where $\mathfrak{T}^j$ denotes the torsion subgroup of $H^j(M, \ZZ)$. 
\end{cor} 
\begin{proof} Applications of the Lefschetz hyperplane-section theorem imply that $H_2 (M,\ZZ)=\ZZ$ 
in each case. On the other hand, $\mathfrak{T}^3 \cong \mathfrak{T}_2$ by the universal coefficients theorem, while 
$\mathfrak{T}^3 \cong \mathfrak{T}_3$ and $\mathfrak{T}^4 \cong \mathfrak{T}_2$ by Poincar\'e duality.
\end{proof} 
 
Wall's classification theorem \cite[Theorem 5] {wall6}  thus allows one to easily determine whether any other given 
smooth compact $6$-manifold is  diffeomorphic to one of these Fano manifolds. This  result specifically concerns closed, spin, simply-connected 
$6$-manifolds $M$ with $\mathfrak{T}_2=0$. The remaining invariants of such a space
are then
\begin{itemize} 
\item the even integer $b_3(M)=\dim H^3(M, \RR)$;
\item  the symmetric trilinear form 
$$H^2 (M,\ZZ) \times H^2(M, \ZZ) \times H^2 (M, \ZZ) \to \ZZ$$
defined by $(\mathsf{a},\mathsf{b},\mathsf{c}) \mapsto \langle \mathsf{a}\cup \mathsf{b} \cup \mathsf{c} , [M]\rangle$; and 
\item the first Pontrjagin class $p_1(M)$, as conveniently encoded by the linear map 
$H^2(M, \ZZ) \to \ZZ$ defined  by $\mathsf{a}\mapsto p_1 \cdot \mathsf{a} := \langle p_1(M) \cup  \mathsf{a}, [M]\rangle$. 
\end{itemize} 
Wall's theorem then says that two closed simply-connected spin $6$-manifolds with torsion-free homology are
 diffeomorphic iff they have the same such invariants relative to some isomorphism of their second cohomology groups.

For our purposes,  it now suffices  to focus on  the special  class of  simply-connected   $6$-manifolds with $H_2(M, \ZZ) \cong \ZZ$ and 
non-trivial trilinear form. Given an orientation of the manifold, 
 there is then a preferred generator $\mathsf{H}\in H^2(M, \ZZ)$ such that $\mathsf{H}^3 > 0$. The invariants in Wall's theorem
 are then encoded  by the three integers $b_3$, $\mathsf{H}^3$, and $p_1\cdot \mathsf{H}$. 
For example, one has $\mathsf{H}= c_1/2$, and hence $\mathsf{H}^3 = c_1^3 /8$, for each of the index-$2$ 
Fano manifolds of Proposition \ref{quidnunc}. At the same time, since any 
Fano $3$-fold has  Todd genus $1=\chi (M, \mathcal{O}) = c_1c_2/24$, the 
first Pontrjagin class $p_1 = c_1^2- 2c_2$ must therefore satisfy 
$$p_1 \cdot \mathsf{H} = \frac{c_1 \cdot p_1}{2}  = \frac{c_1^3 }{2} - c_1c_2= 4 \mathsf{H}^3 - 24.$$
Comparing invariants  with putative Calabi-Yau manifolds, we thus obtain:  

\begin{prop} \label{nuncle} 
A simply-connected Calabi-Yau $3$-fold $(M,J,[\omega ])$ is diffeomorphic to one of the Fano $3$-folds of Propostion \ref{quidnunc} if and only
if it fulfills  the  list of  properties   on the corresponding  row of the following table, 
 \begin{center}
 \begin{tabular}{|c|c|c|c|c|c|}
\hline
Partner ID &  $\mathsf{H}^3$  & $c_2 \cdot \mathsf{H}$ & $h^{1,1}$   & $h^{1,2}$&$\mathfrak{T}_2$ \\
   \hline 
I-11   &  1&  10& 1&20&0\\
   \hline   
   I-12&  2& 8 &1& 9&0 \\
     \hline 
 I-13 &  3& 6 &1&4&0\\
   \hline  
   I-14 &  4& 4&1&1&0\\
\hline
\end{tabular}
\end{center}
where  $\mathsf{H}\in H^2(M, \ZZ)$ is  indivisible over $\ZZ$, and is  a positive multiple of the 
K\"ahler class $[\omega]\in H^2(M,\RR)$.
\end{prop}
\begin{proof} For any such Calabi-Yau, one has $p_1=-2c_2$, and $b_3 = 2h^{1,2}+2$. 
\end{proof} 
There does not seem to be any obvious reason why  these  Calabi-Yaus  could not exist.
However, while there are    constructions \cite{gross,nhlee1,nhlee2} of many Calabi-Yau $3$-folds with $h^{1,1}=1$,
the  ones  sought by Proposition \ref{nuncle} cannot, for example,  arise as complete
intersections or branched covers. Indeed, 
Hirzebruch-Riemann-Roch would predict that  $h^0(M, \mathcal{O} (\mathsf{H}))=\chi (M, \mathcal{O} (\mathsf{H}))=
(2\mathsf{H}^3 + c_2 \cdot \mathsf{H})/12=1$, 
so the  linear system $|\mathsf{H}|$ would therefore consist of a single divisor. Entirely new  ideas
would therefore seem to be required to prove or disprove  the existence of  the   Calabi-Yau $3$-folds in question.
\pagebreak 

%
%

 \vfill 
  
\noindent
{\bf Author's address:}

\bigskip

\noindent
Department of Mathematics\\
Stony Brook University\\
Stony Brook, NY 11794-3651\\ USA

\end{document}